\documentclass{amsart}

\usepackage{lineno}
\usepackage{amssymb}
\usepackage{amsmath}
\usepackage{amsthm}
\usepackage{mathrsfs}
\usepackage{bbm}
\usepackage{cite}
\usepackage{color}
\usepackage{pgfplots}
\usepackage{enumerate}

\graphicspath{{figures/}}
\usepackage{colortbl}

\usepackage[colorlinks,hypertexnames=false]{hyperref}
\usepackage[msc-links]{amsrefs}

\usepackage{changes}

\usepackage{mathtools}
\usepackage{extarrows}
\usepackage{setspace}

\usepackage[misc]{ifsym}

\hypersetup{
	colorlinks=true,
	linkcolor=black,
	filecolor=black,
	urlcolor=black,
	citecolor=black,
}

\newtheorem{thm}{Theorem}[section]
\newtheorem{cor}[thm]{Corollary}
\newtheorem{defn}[thm]{Definition}
\newtheorem{lem}[thm]{Lemma}
\newtheorem{prop}[thm]{Proposition}

\newtheorem{remark}[thm]{Remark}

\numberwithin{equation}{section}

\begin{document}

\title[Star-product of path-slice functions]{ Path-slice  star-product   on non-axially symmetric domains in several quaternionic variables}
\author{Xinyuan Dou}
\email[Xinyuan Dou]{douxinyuan@ustc.edu.cn}
\address{Department of Mathematics, University of Science and Technology of China, Hefei 230026, China}
\address{Institute of Mathematics, AMSS, Chinese Academy of Sciences, Beijing 100190, China}
\author{Ming Jin\textsuperscript{\Letter}}
\email[Ming Jin]{mjin@must.edu.mo}
\address{Faculty of Innovation Engineering, Macau University of Science and Technology, Macau, China} 
\author{Guangbin Ren}
\email[Guangbin Ren]{rengb@ustc.edu.cn}
\address{Department of Mathematics, University of Science and Technology of China, Hefei 230026, China}
\author{Ting Yang}
\email[Ting Yang]{tingy@aqnu.edu.cn}
\address{School of Mathematics and Physics, Anqing Normal University, Anqing 246133, China}
\keywords{several quaternionic variables; path-slice functions; $*$-product;  non-axially symmetric domains; slice topology}
\thanks{This work was supported by the China Postdoctoral Science Foundation (2021M703425) and the NNSF of China (12171448).}

\subjclass[2020]{Primary: 30G35; Secondary: 32A30}

\begin{abstract}
This paper extends the $*$-product from slice analysis to weakly slice analysis in  several  quaternionic variables, focusing on non-axially symmetric domains. It diverges from traditional applications in axially symmetric domains to address slice regularity in more complicated cases. The approach involves redefining the $*$-product for path-slice functions, borrowing techniques from strongly slice analysis. Key to this work is the introduction of relative stem-preserving set pairs and real-path-connected sets, which help establish a direct link between path-slice functions and their stem functions. The study culminates in conditions under which weakly slice regular functions form an algebra in specific slice domains, broadening the scope of slice analysis.
\end{abstract}

\maketitle

\section{Introduction}

The field of slice analysis, particularly concerning quaternionic variables, has experienced a significant transformation. Originally grounded in Euclidean topology \cite{Gentili2007001}, it is now progressively adopting a more intricate slice topology, as indicated in recent studies \cite{Dou2023001}. This transition from a traditionally point-based framework to one that is path-dependent represents a pivotal shift in the field, enabling the exploration of more complicated mathematical terrains. A crucial aspect of this evolution is the introduction of the $*$-product in non-axially symmetric domains, which not only challenges established methodologies but also paves the way for innovative research avenues.

Tracing the history of this development, we find its roots in Hamilton's introduction of quaternions in 1843. This innovation expanded the complex number system and laid the foundation for advancements in quaternionic analysis.
 Initially, the field was notably influenced by Dirac's theory, which focused on linearizing  the relativistic wave  and led to the theoretical prediction of 
 the existence of particles identical in mass to electrons but with opposite electric charges.  This line of research eventually evolved into Atiyah-Singer's spin geometry, marking a major milestone.

 Early work on quaternionic analysis, particularly influenced by Fueter's theory   \cite{Fueter1934001}, revolved around the elliptic version of the Dirac operator. However, this approach faced challenges in handling polynomials, highlighting the need for more comprehensive methods to analyze quaternionic functions.
 
 A major advancement was made by Gentili and Struppa \cite{Gentili2007001}. They transformed slice analysis by adapting the theory of holomorphic functions from complex variables to the quaternionic setting. Their methodology reimagined quaternionic space $\mathbb{H}$ as a collection of complex planes, each distinguished by its own imaginary unit. This perspective allowed for a more nuanced analysis of slice regular functions, particularly those adhering to the Cauchy-Riemann equations.
 
 This innovative approach significantly enhanced the understanding of quaternionic functions and spurred further exploration in the field. Notable areas of progress included quaternionic Schur analysis \cite{Alpay2012001}, quaternionic operator theory \cite{Alpay2015001, MR3887616, MR3967697,Gantner2020001,MR4496722}, and the extension of slice analysis to more generalized structures like real Clifford algebras, octonions, real alternative $*$-algebras and the most general setting of even-dimensional Euclidean space \cite{Colombo2009002,Gentili2010001,Ghiloni2011001,Dou2023001} in one variable, and \cite{Colombo2012002,Dou2023002,Ghiloni2012001,Ghiloni2020001} in several variables.

Recent developments recognize the natural topology for slice analysis of quaternions as the slice topology \cite{Dou2023001}. This understanding has broadened the study of slice regular functions on some sets including non-axially symmetric domains, addressing the complexities of path-dependent functions. Such advancements mark a significant progression in weakly slice analysis, moving away from traditional point-based theories to embrace path-slice functions, which open new perspectives in domains of slice regularity and slice analysis in several variables.

The field distinguishes two types of slice regular functions: weakly slice regular functions, introduced by Gentili and Struppa \cite{Gentili2007001}, and strongly slice regular functions, introduced by Ghiloni and Perotti  \cite{Ghiloni2011001}. The latter expanded the concept to quadratic cones in real alternative $*$-algebras using stem functions, crucial for defining slice functions and extracting properties like holomorphy and multiplication.

The $*$-product of slice regular functions, was firstly introduced in \cite{Gentili2008001} on axially symmetric domains. The $*$-product of slice regular functions on non-symmetric domains in the Euclidean topology of $\mathbb{H}$ is studied in \cite{MR4182982}. In this paper, we present a new approach in the field of slice analysis by innovating the application of the $*$-product to weak slice analysis in the context of several variables, particularly focusing on non-axially symmetric domains. Our primary innovation lies in redefining the $*$-product by harnessing stem functions within the realm of path-slice functions.  

Central to our approach is the introduction of the concept of relative stem-preserving set pairs. We consider two subsets, $\Omega_1$ and $\Omega_2$, within $\mathbb{H}_s^n$. The uniqueness of $\Omega_1$ stems from its real-path-connected nature, ensuring connectivity between any two points within it via paths originating in $\mathbb{R}^n$. In contrast, $\Omega_2$ exhibits an $\Omega_1$-stem-preserving characteristic, denoting a specific form of compatibility with $\Omega_1$. Within this framework, we introduce a path-slice function $f$ that maps from $\Omega_2$ to $\mathbb{H}$, emphasizing its unique properties in relation to the path-structure of its domain.

A key result of our research is the development and validation of a novel function, $F_{\Omega_1}^f$. This function acts on paths within $\mathbb{C}^n$ that are associated with $\Omega_1$, mapping them to a two-dimensional quaternionic vector space, $\mathbb{H}^{2\times 1}$. The design of $F_{\Omega_1}^f$ draws inspiration from Fueter's construction techniques. Furthermore, we explore the scenario where another function, $F$, acts as a path-slice stem function of $f$. In this context, we demonstrate that $F_{\Omega_1}^f$ and $F$ align identically when $F$ is confined to paths in $\mathbb{C}^n$ associated with $\Omega_1$. This finding underscores the equivalence of the two functions under specified conditions.

Our research also transitions from focusing on path-dependent functions to point-based functions, highlighting their intrinsic connection. This transition mirrors the evolution from extension problems in holomorphic functions to functions based on paths, with point-based functions offering advantages in handling slice regular functions.

At the heart of our novel approach is the definition of $\mathscr{F}_{\Omega_1}^f$. This function maps inputs from $\Omega_1$ to a two-dimensional space within $\mathbb{H}^{2\times 1}$. The definition of $\mathscr{F}_{\Omega_1}^f$ varies depending on whether the input $q$ is real or not, showcasing the function's versatility and adaptability to different input scenarios. Our objective is to establish that $\mathscr{F}_{\Omega_1}^f$ is well-defined, given the constraints and properties of the involved sets and functions.

Lastly, addressing the intricacies of the $*$-product, we introduce the concept of self-path-preserving domains. These domains play a pivotal role similar to axially symmetric domains in conventional slice analysis but are specifically tailored for non-axially symmetric domains where the usual bijection between stem and slice functions does not hold. This innovative approach indicates a substantial deviation from traditional methodologies. Furthermore, we show that within any self-stem-preserving domain of the slice cone in several quaternionic spaces, the class of path-slice functions forms an associative unitary real algebra under the $*$-product.

The structure of the paper is as follows:
Section \ref{sc-pre} revisits essential concepts in the weakly slice analysis of several variables, including topics like n-dimensional weakly slice cones and path-slice functions.
In Section \ref{sec-path slice stem function}, the focus is on establishing conditions under which a path-slice function's stem function can be represented by the path-slice function itself. This section introduces concepts such as relative stem-preserving set pairs and real-path-connected sets.
Section \ref{sec-star product} defines the $*$-product for path-slice functions, building upon the findings of Section \ref{sec-path slice stem function}. Section \ref{sec-self stem} delves into the properties of the $*$-product and outlines the criteria for when weakly slice regular functions within a slice-open set can form an algebra.

\section{Preliminaries}\label{sc-pre}

This section revisits essential concepts and definitions in the theory of slice regular functions.

 We denote $\mathbb{H}$ as the algebra of quaternions and define
\begin{equation*}
	\mathbb{S}:=\{I\in\mathbb{H}:I^2=-1 \}
\end{equation*}
as the set of quaternionic imaginary units. Considering $n\in\mathbb{N}$, we define the $n$-dimensional weakly slice cone as
\begin{equation*}
	\mathbb{H}_s^n:=\bigcup_{I\in\mathbb{S}}\mathbb{C}_I^n,
\end{equation*}
where $$\mathbb{C}_I:=\mathbb{R}+\mathbb{R}I  \quad  \mbox{and}\quad \mathbb{C}_I^n:=(\mathbb{C}_I)^n.$$   The slice topology is then described as
\begin{equation*}
	\tau_s(\mathbb{H}^n_s):=\{\Omega\subset \mathbb{H}^n_s\ :\Omega_I\in\tau(\mathbb{C}_I^n),\ \forall\ I\in\mathbb{S}\},
\end{equation*}
where $$\Omega_I:=\Omega\cap\mathbb{C}_I^n$$ and 
 $\tau(\mathbb{C}_I^n)$ denotes the classical topology of $\mathbb{C}_I^n$  when we identify $ \mathbb{C}_I^n $ canonically with $\mathbb C^n$.

Open sets, connected sets, and paths in this topology  $\tau_s(\mathbb{H}^n_s)$  are referred to as slice-open sets, slice-connected sets, and slice-paths in $\Omega$, respectively.

For each $I\in\mathbb{S}$, an isomorphism 
$\Psi_i^I: \mathbb{C}^n \xlongrightarrow[ ]{} \mathbb{C}_I^n$
is defined by 
$$	\Psi_i^I(x+yi)=x+yI,$$
for any $x, y\in \mathbb R^n$.

We denote by $\mathscr{P}(\mathbb{C}^n)$ the set of continuous paths  in $\mathbb{C}^n$ with their initial points in $\mathbb{R}^n$, i.e., 
\begin{equation*}
	\mathscr{P}(\mathbb{C}^n):=\{\gamma\in C([0,1], {\mathbb C}^n):  \gamma(0)\in\mathbb{R}^n \}.
\end{equation*}
For a subset $\Omega$ within $\mathbb{H}_s^n$, we define  
\begin{equation*} \mathscr{P}\left(\mathbb{C}^n,\Omega\right):=\left\{\delta\in\mathscr{P}\left(\mathbb{C}^n\right):\exists\ I\in\mathbb{S},\mbox{ such that } \delta^{I}\subset\Omega\right\},
\end{equation*} 
where 
$$\delta^I := \Psi_i^I\circ\delta$$ represents a path within $\Omega_I$. 

Geometrically, $\mathscr{P}\left(\mathbb{C}^n,\Omega\right)$ consists of all paths    $\mathscr{P} (\mathbb{C}^n)$ in the space $\mathbb C^n$ such that it exists  at least one  lift from 
$\mathbb C^n$ to some slice $\mathbb C^n_I$ which lies entirely in $\Omega$.

Given a fixed path $\gamma \in \mathscr{P} (\mathbb{C}^n, \Omega)$, we then define
 \begin{equation*}
 	{\mathbb{S}(\Omega,\gamma)}:=\left\{I\in\mathbb{S}: \gamma^{I}\subset\Omega\right\} 
 \end{equation*}
to record the  imaginary units   for which the lift path $\gamma^I$ exists in the slice $\mathbb C^n_I$ and  the path $\gamma^I$  lies within $\Omega$.

\begin{defn}
	For a subset $\Omega\subset\mathbb{H}_s^n$, a function $f:\Omega\rightarrow\mathbb{H}$ is called \textit{\textbf{path-slice}} if there exists a function $$F:\mathscr{P}(\mathbb{C}^n,\Omega)\rightarrow \mathbb{H}^{2\times 1}$$ satisfying
	\begin{equation}\label{eq-fcg}
		f\circ\gamma^{I}(1)=(1,I)F(\gamma),
	\end{equation}
	for any $\gamma\in\mathscr{P}(\mathbb{C}^n,\Omega)$ and $I\in\mathbb{S}(\Omega,\gamma)$.
	
	 We refer to $F$ as a path-slice stem function of $f$. The class of path-slice functions defined on $\Omega$ is denoted by $\mathcal{PS}(\Omega)$, and the class of path-slice stem functions of $f\in\mathcal{PS}(\Omega)$ is denoted by $\mathcal{PSS}(f)$.
\end{defn}

\begin{defn}
	For a subset $\Omega$ in the slice topology $\tau_s\left(\mathbb{H}_s^n\right)$, a function $f:\Omega\rightarrow\mathbb{H}$ is termed \textit{\textbf{(weakly) slice regular}} if for each $I\in\mathbb{S}$, the restriction $f_I:=f|_{\Omega_I}$ is (left $I$-)holomorphic. That is, $f_I$ is real differentiable and satisfies
	\begin{equation*}
		\frac{1}{2}\left(\frac{\partial}{\partial x_\ell}+I\frac{\partial}{\partial y_\ell}\right)f_I(x+yI)=0,\qquad\mbox{on}\qquad\Omega_I.
	\end{equation*}
	The class of weakly slice regular functions defined on $\Omega$ is denoted by $\mathcal{SR}(\Omega)$.
\end{defn}

It is established that functions which are weakly slice regular can be classified within the category of path-slice functions \cite[Corollary 5.9]{Dou2023002}.

\section{Path-slice  stem functions}\label{sec-path slice stem function}

In this section, we explore the properties of path-slice stem functions, which are fundamental to the definition of the star-product for path-slice functions.

In our research, it is frequently necessary for the domain we are considering to be real-path-connected.

\begin{defn}\label{def-osmh}
	A subset $\Omega\subset\mathbb{H}_s^n$ is termed \textit{\textbf{real-path-connected}} if, for any point $q\in\Omega$, there exists a path $\gamma\in\mathscr{P}(\mathbb{C}^n,\Omega)$ and an imaginary unit $I\in\mathbb{S}(\Omega,\gamma)$ such that $\gamma^I(1) = q$.
\end{defn}

Recall that  $\mathcal{PS}(\Omega)$ represents the class of path-slice functions, while $\mathcal{PSS}(f)$ denotes the class of path-slice stem functions corresponding to a function $f$ within $\mathcal{PS}(\Omega)$, with these functions being defined on the domain $\Omega$.

\begin{prop}\label{pr-spsfeq}
	Suppose $\Omega\subset\mathbb{H}_s^n$ is real-path-connected. If $f,g\in\mathcal{PS}(\Omega)$ and they share a common path-slice stem function $F\in\mathcal{PSS}(f)\cap \mathcal{PSS}(g)$, then $f$ and $g$ are identical functions, i.e., $f = g$.
\end{prop}

\begin{proof}
	For any point $q\in\Omega$, the real-path-connected nature of $\Omega$ guarantees the existence of a path $\gamma\in\mathscr{P}(\mathbb{C}^n,\Omega)$ and an imaginary unit $I\in\mathbb{S}(\Omega,\gamma)$ such that $\gamma^I(1) = q$. Applying equation \eqref{eq-fcg}, we find
	\begin{equation*}
		f(q) = f\circ\gamma^I(1) = (1,I)F(\gamma) = g\circ\gamma^I(1) = g(q),
	\end{equation*}
	which establishes that $f$ and $g$ coincide.
\end{proof}

\begin{prop}
	Consider a subset $\Omega\subset\mathbb{H}_s^n$ and a function $f\in\mathcal{PS}(\Omega)$. Let $\gamma\in\mathscr{P}(\mathbb{C}^n,\Omega)$, $F$ be a path-slice stem function of $f$, and $I,J\in\mathbb{S}(\Omega,\gamma)$ with $I\neq J$. It follows that
	\begin{equation}\label{eq-fgbp}
		F(\gamma) = \begin{pmatrix}
			1&I\\ 1& J
		\end{pmatrix}^{-1}\begin{pmatrix}
			f\circ\gamma^I(1)\\ f\circ\gamma^J(1)
		\end{pmatrix}.
	\end{equation}
\end{prop}

\begin{proof}
	Simple calculations validate equation \eqref{eq-inverse}. Applying equation \eqref{eq-fcg}, we have
	\begin{equation*}
		\begin{pmatrix}
			f\circ\gamma^I(1)\\ f\circ\gamma^J(1)
		\end{pmatrix}
		= \begin{pmatrix}
			1&I\\1& J
		\end{pmatrix}F(\gamma),
	\end{equation*}
	implying the validity of equation \eqref{eq-fgbp}.
\end{proof}

 \begin{remark}
 	The inverse matrix in \eqref{eq-fgbp} is given by
 \begin{equation}\label{eq-inverse}
 	\begin{pmatrix}
 		1&I\\ 1& J
 	\end{pmatrix}^{-1} = (I-J)^{-1}\begin{pmatrix}
 		I & -J\\ 1 & -1
 	\end{pmatrix}.
 \end{equation}
 \end{remark}
 
We define a particular set of pair of paths within $\mathbb{C}^n$  with the same end point as follows:
\begin{equation}\label{eq-mp*2}
	\mathscr{P}_*^2(\mathbb{C}^n,\Omega) 
	 :=\left\{
	 	(\alpha,\beta) 
	 	\in \big(\mathscr{P}(\mathbb{C}^n,\Omega)\big)^2: 
		\alpha(1) = \beta(1)
	\right\}.
\end{equation}

\begin{lem}\label{lem-1}
	Given $\Omega\subset\mathbb{H}_s^n$, a function $f\in\mathcal{PS}(\Omega)$, and a path-slice stem function $F\in\mathcal{PSS}(f)$ of $f$, for any pair $(\alpha,\beta)$ in $\mathscr{P}_*^2(\mathbb{C}^n,\Omega)$ with at least two shared imaginary units in $\mathbb{S}(\Omega,\alpha)$ and $\mathbb{S}(\Omega,\beta)$, it holds that
	\begin{equation}\label{eq-stemcoincide}
		F(\alpha) = F(\beta).
	\end{equation}
\end{lem}

\begin{proof}
	Considering distinct $I, J$ from the intersection $\mathbb{S}(\Omega,\alpha) \cap \mathbb{S}(\Omega,\beta)$, the equality $\alpha(1) = \beta(1)$ leads to
	\begin{equation}\label{eq-al1}
		\alpha^L(1) = \beta^L(1)
	\end{equation}
for each $L \in \{I, J\}.$
	Based on equation \eqref{eq-fgbp} and the implication from \eqref{eq-al1}, it follows that
	\begin{eqnarray*}
		F(\alpha) &=& \begin{pmatrix} 1 & I \\ 1 & J \end{pmatrix}^{-1} \begin{pmatrix} f\circ\alpha^I(1) \\ f\circ\alpha^J(1) \end{pmatrix} 
				\\
		&=&\begin{pmatrix} 1 & I \\ 1 & J \end{pmatrix}^{-1} \begin{pmatrix} f\circ\beta^I(1) \\ f\circ\beta^J(1) \end{pmatrix}
		\\
		&=&  F(\beta).
	\end{eqnarray*}
\end{proof}

In our exploration of the $*$-product, it is   necessary to introduce a novel category of domains, termed as  stem-preserving domains.relative stem-preserving set pairs.

\begin{defn}\label{defn-lo1o2}
	For subsets $\Omega_1, \Omega_2 \subset \mathbb{H}_s^n$, $\Omega_2$ is \textit{\textbf{$\Omega_1$-stem-preserving}} if it satisfies:
	\begin{enumerate}[\upshape (i)]
		\item\label{it-lmo2} Every $\gamma \in \mathscr{P}(\mathbb{C}^n,\Omega_1)$ has at least two corresponding imaginary units in $\mathbb{S}(\Omega_2,\gamma)$, i.e.,
		$$|\mathbb{S}(\Omega_2,\gamma)|\geqslant 2. $$

		\item\label{it-lmso} For any $(\alpha,\beta) \in \mathscr{P}_*^2(\mathbb{C}^n,\Omega_1)$, $$\left|\mathbb{S}(\Omega_2,\alpha) \cap \mathbb{S}(\Omega_2,\beta)\right| \neq 1.$$
	\end{enumerate}
\end{defn}

\begin{lem}
Let  $\Omega_1$ and $\Omega_2$  be  subsets of $\mathbb{H}_s^n$. If   $\Omega_2$ is $\Omega_1$-stem-preserving, then
	\begin{equation}\label{eq-mpmn}
		\mathscr{P}(\mathbb{C}^n,\Omega_1) \subseteq \mathscr{P}(\mathbb{C}^n,\Omega_2).
	\end{equation}
\end{lem}

\begin{proof}
	Take any path $\gamma$ from the set $\mathscr{P}(\mathbb{C}^n,\Omega_1)$. Given that the initial point $\gamma(0)$ lies in $\mathbb{R}^n$, and considering that there are at least two elements in $\mathbb{S}(\Omega_2,\gamma)$ (as per Definition \ref{defn-lo1o2} \eqref{it-lmo2}), we select an imaginary unit $I$ from $\mathbb{S}(\Omega_2,\gamma)$. This selection leads to the conclusion that the path $\gamma^I$ falls within the subset $\Omega_2$. Therefore, $\gamma$ qualifies as an element of $\mathscr{P}(\mathbb{C}^n,\Omega_2)$. This confirmation aligns with the assertion in equation \eqref{eq-mpmn}.
\end{proof}

In contrast to the classical theory in axially-symmetric domains, for any path-slice function, its corresponding stem functions may be not unique. However, if
 $\Omega_2$ is $\Omega_1$-stem-preserving, then under certain conditions we can show that it is unique when restricted to some  subset of the definition domain.
Such a unique function, we call it sub-stem function of the path-slice function  $f$.

  \begin{defn}  Suppose  $\Omega_1$ and $\Omega_2$ are  subsets of $\mathbb{H}_s^n$. 
   $\Omega_1$ is real-path-connected,  $\Omega_2$ is $\Omega_1$-stem-preserving, and $f: \Omega_2 \rightarrow \mathbb{H}$ is path-slice.  We define  the function
  	\begin{equation}\label{eq-F def}
  		F_{\Omega_1}^f:  \mathscr{P}(\mathbb{C}^n,\Omega_1) \rightarrow \mathbb{H}^{2\times 1}
  	\end{equation}
  as
  	\begin{equation}\label{eq:almost-stem-2091}
  		F_{\Omega_1}^f(\gamma) := \begin{pmatrix} 1 & I \\ 1 & J \end{pmatrix}^{-1} \begin{pmatrix} f\circ\gamma^I(1) \\  f\circ\gamma^J(1). \end{pmatrix} 
  	\end{equation}
  Here   $I,J$ are arbitrary elements in $\mathbb{S}(\Omega_2,\gamma)$ with  the condition that   $I \neq J$.  We call this function the sub-stem function of $f$. 
  	  \end{defn}

  We need to show that \eqref{eq:almost-stem-2091} is well defined, i.e., independent of the choice of $I$ and $J$. We also show our claim that the path-stem function is unique restricted on the subset 
  $\mathscr{P}(\mathbb{C}^n,\Omega_1)$, which is nothing but the sub-stem function.

\begin{prop}\label{eq:234} Let  $\Omega_1$ and $\Omega_2$  be  subsets of $\mathbb{H}_s^n$. 
	If $\Omega_1$ is real-path-connected,  $\Omega_2$ is $\Omega_1$-stem-preserving, and $f: \Omega_2 \rightarrow \mathbb{H}$ is path-slice, then the function
	\begin{equation}\label{eq-F def}
		F_{\Omega_1}^f:  \mathscr{P}(\mathbb{C}^n,\Omega_1) \rightarrow \mathbb{H}^{2\times 1},
	\end{equation}
	defined in  \eqref{eq:almost-stem-2091} is well-defined.

	Moreover, if $F$ is a path-slice stem function of $f$, then
	\begin{equation}\label{eq-Ffomega=F}
		F_{\Omega_1}^f = F|_{\mathscr{P}(\mathbb{C}^n,\Omega_1)}.
	\end{equation}
\end{prop}

\begin{proof}
	For any $I, J, I', J' \in \mathbb{S}(\Omega_2,\gamma)$ with $I \neq J$ and $I' \neq J'$, from \eqref{eq-fgbp} we have
	\begin{equation*}
		\begin{pmatrix} 1 & I \\ 1 & J \end{pmatrix}^{-1} \begin{pmatrix} f\circ\gamma^I(1) \\ f\circ\gamma^J(1) \end{pmatrix} = F(\gamma) = \begin{pmatrix} 1 & I' \\ 1 & J' \end{pmatrix}^{-1} \begin{pmatrix} f\circ\gamma^{I'}(1) \\ f\circ\gamma^{J'}(1) \end{pmatrix},
	\end{equation*}
	confirming that $F_{\Omega_1}^f$  is independent of the specific choice of $I, J$, 
	and thus, $F_{\Omega_1}^f$ 
is well-defined. 
	It also verifies  \eqref{eq-Ffomega=F}.
 \end{proof}

 For any point $q = (q_1, \ldots, q_n)$ in $\mathbb{H}_s^n$, it can be located within a specific slice $\mathbb{C}_I^n$ for some $I\in \mathbb{S}$. However, it is important to note that
$$\mathbb C_I^n=\mathbb C_{-I}^n.$$
 To address this, we assign a specific imaginary unit.         The function $\mathfrak{I}$ is defined as:
 \begin{equation*}
 	\mathfrak{I}:\ \mathbb{H}_s^n \rightarrow \mathbb{S} \cup \{0\},
 \end{equation*}
 which maps each point $q$ to 0 if $q$ is in $\mathbb{R}^n$. If $q$ is not in $\mathbb{R}^n$, $\mathfrak{I}(q)$ is defined as $$\frac{q_\imath - \mbox{Re}(q_\imath)}{|q_\imath - \mbox{Re}(q_\imath)|},$$ where $\imath \in \{1,\ldots, n\}$ is the smallest positive integer for which $q_\imath$ is not a real number.

We now demonstrates that in a real-path-connected subset of the weakly slice cone, for any non-real point, there exists a path whose image under a specific imaginary unit mapping lies within the subset and ends at that point.

 \begin{lem}
 	If $\Omega\subset\mathbb{H}_s^n$ is real-path-connected and $q\in\Omega \setminus \mathbb R^n$, then there exists a path $\gamma\in\mathscr{P}(\mathbb{C}^n,\Omega)$ such that
 	\begin{equation}\label{eq-frakiq}
 		\gamma^{\mathfrak{I}(q)}\subset\Omega,  \quad  \gamma^{\mathfrak{I}(q)}(1)=q.
 	\end{equation}
 \end{lem}
 
 \begin{proof}
 	Since $\Omega$ is real-path-connected, there is a path $\beta\in\mathscr{P}(\mathbb{C}^n,\Omega)$ and an imaginary unit $I\in\mathbb{S}$ such that $\beta^I\subset\Omega$ and $\beta^I(1)=q$. If $I$ is not equal to $\pm\mathfrak{I}(q)$, then $q$ would belong to $\mathbb{R}^n$, which contradicts the assumption. Therefore, $I$ must be $\pm\mathfrak{I}(q)$. We define $\gamma$ as
 	\begin{equation}\label{eq-bbco}
 		\gamma := \begin{cases}
 			\overline{\beta}, &\text{ if } I \neq \mathfrak{I}(q),
 			\\ \beta, &\text{ otherwise},
 		\end{cases}
 	\end{equation}
 	Here 
 the quaternion conjugation is defined by the formula 	$$\overline{x_0+x_1 i+x_2 j+x_3k} := x_0-x_1 i-x_2 j-x_3k,  $$ 
  where $i, j, k$ represent the standard basis vectors in quaternionic space.

 Since $I=\pm\mathfrak{I}(q)$,  
 it follows that
 	\begin{equation*}
 		\gamma^{\mathfrak{I}(q)} = \begin{cases}
 			\overline{\beta}^{-I} = \beta^I, &\text{ if } I \neq \mathfrak{I}(q),
 			\\ \beta^I, &\text{ otherwise},
 		\end{cases}
 	\end{equation*}
 	ensuring $\gamma^{\mathfrak{I}(q)}\subset\Omega$ and satisfying \eqref{eq-frakiq}.
 \end{proof}

 The subsequent  construction demonstrates the transition from functions that depend on paths to the more traditional functions based on specific points.

 \begin{prop}
 	If $\Omega_1\subset\mathbb{H}_s^n$ is real-path-connected, $\Omega_2\subset\mathbb{H}_s^n$ is $\Omega_1$-stem-preserving, and $f:\Omega_2\rightarrow\mathbb{H}$ is a path-slice function, then the function 
 	\begin{equation}\label{eq-mathscr f def}
 		\mathscr{F}_{\Omega_1}^f:   \Omega_1 \rightarrow \mathbb{H}^{2\times 1},
 	\end{equation}
 	defined by
 	\begin{equation}\label{eq-mathscr f def-049}
 			\mathscr{F}_{\Omega_1}^f(q)= \begin{cases}
 			F_{\Omega_1}^f(\gamma), &\text{ if } q \notin \mathbb{R}^n,
 			\\(f(q),0)^T, &\text{ otherwise},
 		\end{cases}
 	\end{equation}
 	is well-defined, where $\gamma\in\mathscr{P}(\mathbb{C}^n,\Omega_1)$ satisfies $\gamma^{\mathfrak{I}(q)}\subset\Omega_1$ and $\gamma^{\mathfrak{I}(q)}(1)=q$.
 \end{prop}
 
 \begin{proof}
 	We focus on the case where $q\notin\mathbb{R}^n$. Let $F$ be a path-slice stem function, and consider $\gamma_1, \gamma_2 \in \mathscr{P}(\mathbb{C}^n,\Omega_1)$ such that $$\gamma_1^{\mathfrak{I}(q)}, \gamma_2^{\mathfrak{I}(q)}\subset\Omega_1, \quad \gamma_1^{\mathfrak{I}(q)}(1) = \gamma_2^{\mathfrak{I}(q)}(1) = q.$$   Since $\Omega_2$ is $\Omega_1$-stem-preserving, every $\beta\in\mathscr{P}(\mathbb{C}^n,\Omega_1)$ has at least two corresponding imaginary units in $\Omega_2$. Hence, $$\left|\mathbb{S}(\Omega_2,\gamma_1)\right|\geqslant 2, \quad  \left|\mathbb{S}(\Omega_2,\gamma_2)\right|\geqslant  2.$$ From \eqref{eq-stemcoincide} and \eqref{eq-Ffomega=F}, it follows that
 	\begin{equation*}
 		F_{\Omega_1}^f(\gamma_1) = F|_{\mathscr{P}(\mathbb{C}^n,\Omega_1)}(\gamma_1) = F|_{\mathscr{P}(\mathbb{C}^n,\Omega_1)}(\gamma_2) = F_{\Omega_2}^f(\gamma_2),
 	\end{equation*}
 	confirming that $\mathscr{F}_{\Omega_1}^f$ is well-defined.
 \end{proof}

 \begin{lem} 	
 	Consider two subsets 
 $\Omega_1$ and 
 $\Omega_2$ within the space 
 $\mathbb H^n_s$, where 
 $\Omega_1$
 is  real-path-connected and 
 $\Omega_2$ be $\Omega_1$-stem-preserving. 
 Let  $f:\Omega_2\rightarrow\mathbb{H}$ be path-slice.
   Then,  the following equation holds  	
 \begin{equation}\label{eq-lfo1}  \left.\mathscr{F}_{\Omega_1}^f\right|_{(\Omega_1)_{\mathbb R}}=
 	 \begin{pmatrix} 
 	 	\left. f \right|_{(\Omega_1)_{\mathbb R}}
 	 	 \\ 0\end{pmatrix}.
  	 \end{equation} \end{lem}

 \begin{proof} This follows directly from  \eqref{eq-mathscr f def-049}.
 	\end{proof}

 \begin{lem} 	Consider two subsets 
 	$\Omega_1$ and 
 	$\Omega_2$ within the space 
 	$\mathbb H^n_s$, where 
 	$\Omega_1$
  is  real-path-connected and 
 	$\Omega_2$ be $\Omega_1$-stem-preserving. 
Let  $f:\Omega_2\rightarrow\mathbb{H}$ be path-slice. Then
for any  $\gamma\in\mathscr{P}(\mathbb{C}^n,\Omega_1)$ we have
 	\begin{equation}\label{eq-fgamma conj}
 		F_{\Omega_1}^{f}(\gamma)
 		=\begin{pmatrix}
 			1\\ & -1
 		\end{pmatrix}F_{\Omega_1}^{f}(\overline{\gamma}).
 	\end{equation}
 \end{lem}
 
 \begin{proof}
 	Consider a path 
 $\gamma$ in $\mathscr{P}(\mathbb{C}^n,\Omega_1)$.
    Given that 
	$\Omega_2$ is  
	$\Omega_1$-stem-preserving,
   there exist at least two distinct elements in $\mathbb{S}(\Omega_2,\gamma)$,
   denoted as 
 $I$ and 
 $J$
 	 with 
 	$I\neq J$.
 	 Based on equation \eqref{eq:almost-stem-2091} and the relationship 
 $$\gamma^I=\overline{\gamma}^{-I}, $$
  it follows that 
 		\begin{equation*}
 		\begin{split}
 			F_{\Omega_1}^{f}(\gamma)
 			=&\begin{pmatrix}
 				1&I\\1&J
 			\end{pmatrix}^{-1}\begin{pmatrix}
 				f\circ\gamma^I(1)\\f\circ\gamma^J(1)
 			\end{pmatrix}
 			\\=&\begin{pmatrix}
 				1&I\\1&J
 			\end{pmatrix}^{-1}\begin{pmatrix}
 				1&-I\\1&-J
 			\end{pmatrix}\begin{pmatrix}
 				1&-I\\1&-J
 			\end{pmatrix}^{-1}\begin{pmatrix}
 				f\circ\overline{\gamma}^{-I}(1)\\f\circ\overline{\gamma}^{-J}(1)
 			\end{pmatrix}
 			\\=&\begin{pmatrix}
 				1&I\\1&J
 			\end{pmatrix}^{-1}\begin{pmatrix}
 				1&I\\1&J
 			\end{pmatrix}
 			\begin{pmatrix}
 				1\\ & -1
 			\end{pmatrix}F_{\Omega_1}^{f}(\overline{\gamma})
 			\\=&\begin{pmatrix}
 				1\\ & -1
 			\end{pmatrix}F_{\Omega_1}^{f}(\overline{\gamma}).
 		\end{split}
 	\end{equation*}
 \end{proof}
 
 \begin{lem}
 	Let $\Omega_1\subset\mathbb{H}_s^n$ be real-path-connected, $\Omega_2\subset\mathbb{H}_s^n$ be $\Omega_1$-stem-preserving, $f:\Omega_2\rightarrow\mathbb{H}$ be path-slice, and $\gamma\in\mathscr{P}(\mathbb{C}^n,\Omega_1)$ with $\gamma(1)\in\mathbb{R}^n$. Then
 	\begin{equation}\label{eq-fgamma R}
 		F_{\Omega_1}^{f}(\gamma)
 		=\begin{pmatrix}
 			f\circ\gamma(1)\\0
 		\end{pmatrix}.
 	\end{equation}
 \end{lem}
 
 \begin{proof}
 	Considering the path $\gamma\in\mathscr{P}(\mathbb{C}^n,\Omega_1)$ and the stem-preserving nature of \(\Omega_2\), we find at least two distinct elements in \(\mathbb{S}(\Omega_2, \gamma)\), denoted as \(I\) and \(J\) with \(I \neq J\). The condition \(\gamma(1) \in \mathbb{R}^n\) implies \(\gamma^I(1) = \gamma^J(1)\). According  to equation \eqref{eq-F def}, it can be deduced that
 	  	\begin{equation*}
 		\begin{split}
 			F_{\Omega_1}^{f}(\gamma)
 			=&\begin{pmatrix}
 				1&I\\1&J
 			\end{pmatrix}^{-1}\begin{pmatrix}
 				f\circ\gamma^I(1)\\f\circ\gamma^J(1)
 			\end{pmatrix}
 			\\=&\begin{pmatrix}
 				1&I\\1&J
 			\end{pmatrix}^{-1}\begin{pmatrix}
 				f\circ\gamma^J(1)\\f\circ\gamma^J(1)
 			\end{pmatrix}
 			\\=&\begin{pmatrix}
 				1&I\\1&J
 			\end{pmatrix}^{-1}\begin{pmatrix}
 				1&I\\1&J
 			\end{pmatrix}\begin{pmatrix}
 				f\circ\gamma^J(1)\\ 0
 			\end{pmatrix}
 			\\=&\begin{pmatrix}
 				f\circ\gamma^J(1)\\ 0
 			\end{pmatrix}.
 		\end{split}
 	\end{equation*}
 This finishes the proof. 
 \end{proof}

\section{Star-product for path-slice functions}\label{sec-star product}

In this section, we aim to introduce the concept of the star-product specifically for path-slice functions. Additionally, we will demonstrate that the star-product of two path-slice functions is also path-slice.

\begin{defn}
	Let $\Omega_1\subset\mathbb{H}_s^n$ be real-path-connected, $\Omega_2\subset\mathbb{H}_s^n$ be $\Omega_1$-stem-preserving, $f\in\mathcal{PS}(\Omega_1)$ and $g\in\mathcal{PS}(\Omega_2)$. We define
the function  	$$f*g: \Omega_1\rightarrow\mathbb{H}$$ by 
	\begin{equation*}
		f*g(q):=(f(q),\mathfrak{I}(q)f(q)) \mathscr{F}_{\Omega_1}^{g}(q). 
	\end{equation*}
The operation is called   	the $*$-product of $f$ and $g$.
\end{defn}

\begin{lem}
	Let $\Omega_1\subset\mathbb{H}_s^n$ be real-path-connected, $\Omega_2\subset\mathbb{H}_s^n$ be $\Omega_1$-stem-preserving, $c\in\mathbb{H}$, $g\in\mathcal{PS}(\Omega_2)$, $\gamma\in\mathscr{P}(\mathbb{C}^n,\Omega_1)$, $I\in\mathbb{S}(\Omega_1,\gamma)$ and $q:=\gamma^I(1)$. Then
	\begin{equation}\label{eq-overline}
		(c,Ic) F_{\Omega_1}^{g}(\gamma)=(c,-Ic) F_{\Omega_1}^{g}(\overline{\gamma})=\left(c,\mathfrak{I}(q)c\right) \mathscr{F}_{\Omega_1}^{g}(q).
	\end{equation}
\end{lem}

\begin{proof} Utilizing equation \eqref{eq-fgamma conj},
 we have
	\begin{equation*}
		(c,Ic) F_{\Omega_1}^{g}(\gamma)=(c,Ic)\begin{pmatrix}
			1\\ & -1
		\end{pmatrix}\begin{pmatrix}
			1\\ & -1
		\end{pmatrix} F_{\Omega_1}^{g}(\gamma)
		=(c,-Ic)F_{\Omega_1}^{g}(\overline{\gamma}).
	\end{equation*}
	
	If $q$ does not belong to $\mathbb{R}^n$, then $\mathfrak{I}(q)$ equals  either $I$ or $-I$. In these scenarios, the following cases apply
	\begin{equation*}
		\begin{cases}
			\gamma^{\mathfrak{I}(q)}=\gamma^I,\mbox{ and }\mathscr{F}_{\Omega_1}^{g}(q)=F_{\Omega_1}^{g}(\gamma),\qquad & \mbox{ if }\mathfrak{I}(q)=I,
			\\\overline{\gamma}^{\mathfrak{I}(q)}=\gamma^I,\mbox{ and }\mathscr{F}_{\Omega_1}^{g}(q)=F_{\Omega_1}^{g}(\overline{\gamma}), \qquad & \mbox{ if }\mathfrak{I}(q)=-I,
		\end{cases}
	\end{equation*}
thereby validating \eqref{eq-overline}.
	
	In the case where 
	$q$ is in  $\mathbb{R}^n$, 
  according to \eqref{eq-lfo1} and \eqref{eq-fgamma R}, we deduce 
	\begin{equation*}
		\begin{split}
			(c,Ic) F_{\Omega_1}^{g}(\gamma)
			=&(c,Ic)\begin{pmatrix}
				g\circ\gamma(1)\\0
			\end{pmatrix}
			=c\cdot g\circ\gamma(1)
			\\=&c\cdot g(q)
			=\left(c,\mathfrak{I}(q)c\right)\begin{pmatrix}
				g(q)\\0
			\end{pmatrix}
			=\left(c,\mathfrak{I}(q)c\right) \mathscr{F}_{\Omega_1}^{g}(q).
		\end{split}
	\end{equation*}
\end{proof}

\begin{lem}
	Let $\Omega_1\subset\mathbb{H}_s^n$ be real-path-connected, $\Omega_2\subset\mathbb{H}_s^n$ be $\Omega_1$-stem-preserving, $f\in\mathcal{PS}(\Omega_1)$ and $g\in\mathcal{PS}(\Omega_2)$. Then
	\begin{equation}\label{eq-llfg}
		\left.\left(f*g\right)\right|_{(\Omega_1)_\mathbb{R}}
		=\left.(fg)\right|_{(\Omega_1)_\mathbb{R}}.
	\end{equation}
\end{lem}

\begin{proof} We begin by examining the expression $\left.\left(f*g\right)\right|_{(\Omega_1)_\mathbb{R}}$
	   Expanding this based on the definition of the 
$*$-product, we get 
\begin{equation*}
		\begin{split}
			\left.\left(f*g\right)\right|_{(\Omega_1)_{\mathbb{R}}}
			=&\left.\left[\left(f,\mathfrak{I}f\right) \mathscr{F}_{\Omega_1}^{g}\right]\right|_{(\Omega_1)_{\mathbb{R}}}
			=\left(f|_{(\Omega_1)_{\mathbb{R}}},(\mathfrak{I}f)|_{(\Omega_1)_{\mathbb{R}}}\right)\left.\mathscr{F}_{\Omega_1}^{g}\right|_{(\Omega_1)_{\mathbb{R}}}
			\\=&\left(f|_{(\Omega_1)_{\mathbb{R}}},(\mathfrak{I}f)|_{(\Omega_1)_{\mathbb{R}}}\right)\begin{pmatrix}
				g|_{(\Omega_1)_\mathbb{R}}\\0
			\end{pmatrix}
			=f|_{(\Omega_1)_{\mathbb{R}}}g|_{(\Omega_1)_\mathbb{R}}
			=\left.(fg)\right|_{(\Omega_1)_\mathbb{R}}.
		\end{split}
	\end{equation*}
In this calculation, we have applied the simplification provided by \eqref{eq-lfo1}. With this, the proof is concluded.
\end{proof}

From the canonical identification
$$ \mathbb{H}^{2\times 1} \simeq \mathbb H \otimes_{\mathbb R} \mathbb C,$$
the  standard product from  complexified quaternions   induces the $*$-product in the space $\mathbb{H}^{2\times 1}$,  given  as follows.

For any  $p:=(p_1,p_2)^T,q:=(q_1,q_2)^T\in \mathbb{H}^{2\times 1}$, we define its $*$-product as 
\begin{equation*}
	p*q:=\left(p_1\mathbb{I}+p_2\sigma\right)\left(q_1\mathbb{I}+q_2\sigma\right)e_1,
\end{equation*}
where
\begin{equation*}
	\mathbb{I}:=\begin{pmatrix}
		1\\ & 1
	\end{pmatrix},\qquad
	\sigma:=\begin{pmatrix}
		& -1\\ 1
	\end{pmatrix},  \qquad
	e_1:=\begin{pmatrix}
		1\\0
	\end{pmatrix}.
\end{equation*}

The $*$-product in  $\mathbb{H}^{2\times 1}$ can be extended to functions with values in  $\mathbb{H}^{2\times 1}$.

\begin{defn}\label{def-osmh-236}
Let $\Omega\subset\mathbb{H}_s^n$, and $F,G:\mathscr{P}(\mathbb{C}^n,\Omega)\rightarrow\mathbb{H}^{2\times 1}$. We define
\begin{equation*}
	\begin{split}
		F*G:\ \mathscr{P}(\mathbb{C}^n,\Omega)\ &\xlongrightarrow[\hskip1cm]{}\ \mathbb{H}^{2\times 1} 
		\end{split}
\end{equation*}
by
$$F*G( \gamma)=F(\gamma) * G(\gamma).$$

\end{defn}

We now  examines path-slice functions within quaternionic spaces, specifically in $\Omega_1$ and $\Omega_2$ subsets of $\mathbb{H}_s^n$, and their behaviors under the $*$ operation.

\begin{prop}\label{pr-star product}
	Let $\Omega_1\subset\mathbb{H}_s^n$ be real-path-connected, $\Omega_2\subset\mathbb{H}_s^n$ be $\Omega_1$-stem-preserving, $f\in\mathcal{PS}(\Omega_1)$ and $g\in\mathcal{PS}(\Omega_2)$. Then $f*g$ is path-slice. 
	
	Moreover, if $F$ is a path-slice stem function of $f$, then
		$F*F_{\Omega_1}^g$   is a path-slice stem function of $f*g$.
\end{prop}

\begin{proof}
	We start by taking a path $\gamma$ from   $\mathscr{P}(\mathbb{C}^n,\Omega_1)$. Additionally, we select an element $I$ from the set $\mathbb{S}(\Omega_1,\gamma)$. We define $q$ as   $\gamma^I(1)$.

	According  to equation \eqref{eq-overline}, we analyze the operation involving $(1,I) \left(F*F_{\Omega_1}^g\right)$ applied to a path $\gamma$. This operation is defined as the combination of the components $F_1+F_2\sigma$ and $F_{\Omega_1}^{g,1}+F_{\Omega_1}^{g,2}\sigma$ followed by multiplication with $e_1$.

	 	\begin{eqnarray*}
	 	\begin{split}
	 	  &\left\{(1,I) \left(F*F_{\Omega_1}^g\right)\right\}(\gamma)
	 		\\
	 	 =&\left\{(1,I)\left[F_1+F_2\sigma\right]\left[F_{\Omega_1}^{g,1}+F_{\Omega_1}^{g,2}\sigma\right] e_1\right\}(\gamma)
	 		\\=&\left\{(1,I)\left[\left(F_1\cdot F_{\Omega_1}^{g,1}-F_2\cdot F_{\Omega_1}^{g,2}\right)+\left(F_2\cdot F_{\Omega_1}^{g,1}+F_1\cdot F_{\Omega_1}^{g,2}\right)\sigma\right] e_1\right\}(\gamma)
	 		\\=&\left\{(1,I)\begin{pmatrix}
	 			F_1\cdot F_{\Omega_1}^{g,1}-F_2\cdot F_{\Omega_1}^{g,2}\\
	 			F_2\cdot F_{\Omega_1}^{g,1}+F_1\cdot F_{\Omega_1}^{g,2}
	 		\end{pmatrix}\right\}(\gamma)
	 		 	\end{split}
	 \end{eqnarray*}

	 Ultimately, this complex expression simplifies to  to the concise form $f*g(q)$:
	 	\begin{equation*}
		\begin{split}
		 		&(F_1+IF_2)(\gamma)\cdot F_{\Omega_1}^{g,1}(\gamma)+I(F_1+IF_2)(\gamma)\cdot F_{\Omega_1}^{g,2}(\gamma)
			\\=& f(q)F_{\Omega_1}^{g,1}\left(\gamma\right)+If(q)F_{\Omega_1}^{g,2}\left(\gamma\right)
			\\=&\left(f(q),If(q)\right)F_{\Omega_1}^{g}(\gamma)
			\\=&\left(f(q),\mathfrak{I}(q)f(q)\right) \mathscr{F}_{\Omega_1}^{g}(q)
			\\=&f*g(q).
		\end{split}
	\end{equation*}
	It implies that $F*F_{\Omega_1}^g$ is a path-slice stem function of $f*g$, which implies that $f*g$ is path-slice.
\end{proof}

We now exhibit the associative property of the $*$-product.

\begin{lem}
	Let $\Omega_1\subset\mathbb{H}_s^n$ be real-path-connected, $\Omega_2\subset\mathbb{H}_s^n$ be real-path-connected and $\Omega_1$-stem-preserving, $\Omega_3\subset\mathbb{H}_s^n$ be $\Omega_2$-stem-preserving, $f\in\mathcal{PS}(\Omega_1)$, $g\in\mathcal{PS}(\Omega_2)$ and $h\in\mathcal{PS}(\Omega_3)$. Then
	\begin{equation}\label{eq-associative}
		(f*g)*h=f*(g*h).
	\end{equation}
\end{lem}

\begin{proof}
	Let $F\in\mathcal{PSS}(f)$, $G\in\mathcal{PSS}(g)$, $H\in\mathcal{PSS}(h)$ and $D_\imath:={\mathscr{P}(\mathbb{C}^n,\Omega_\imath)}$, $\imath=1,2,3$. According to Proposition \ref{pr-star product},
	\begin{eqnarray*}
		&&F*G|_{D_1}\in\mathcal{PSS}(f*g), 
		\\
		&&G*H|_{D_2}\in\mathcal{PSS}(g*h).
	\end{eqnarray*}
	Again by Proposition \ref{pr-star product},
	\begin{eqnarray*}
	&&(F*G|_{D_1})*H|_{D_1}\in\mathcal{PSS}((f*g)*h), 
		\\
	&&F*(G*H|_{D_2})|_{D_1}\in\mathcal{PSS}(f*(g*h)).
	\end{eqnarray*}
From Proposition \ref{pr-spsfeq}, we have 
	\begin{eqnarray*}
		(F*G|_{D_1})*H|_{D_1}&=&F*G|_{D_1}*h|_{D_1}
		\\
		&=&F*(G*H|_{D_2})|_{D_1}
	\end{eqnarray*}
so	that $(f*g)*h=f*(g*h)$.
\end{proof}

\section{Self-stem-preserving domains}\label{sec-self stem}
In this section, we delve into the concept of a new kind of domain, referred to as self-stem-preserving. These domains expand upon the idea of axially symmetric domains, establishing a unique relationship between path-slice functions and path-slice stem functions. We will explore how, in these domains, the collection of slice functions constitutes an associative unitary real algebra, particularly when they are combined through the use of the $*$-product. Our discussion includes presenting examples of self-stem-preserving domains. Furthermore, we assert that any slice-domain, provided it is open in the Euclidean sense and intersects with the real axis, can be classified as self-stem-preserving. Lastly, we will demonstrate that every weakly axially symmetric slice-domain inherently displays characteristics of self-stem-preservation.
\begin{defn}
	A subset $\Omega$ of $\mathbb{H}_s^n$ 
	is termed self-stem-preserving if it is real-path-connected and also 
	  $\Omega$-stem-preserving.
\end{defn}

\begin{thm}
	Let $\Omega\subset\mathbb{H}_s^n$ be self-stem-preserving. Then a function $f:\Omega\to \mathbb H$ is a path-slice function if and only if $F_{\Omega}^f:  \mathscr{P}(\mathbb{C}^n,\Omega) \rightarrow \mathbb{H}^{2\times 1}$
	is the unique path-slice stem function of $f$. The relationship between them is expressed by the formula:
	\begin{equation}\label{eq:almost-stem-209}
	F_{\Omega}^f(\gamma) = \begin{pmatrix} 1 & I \\ 1 & J \end{pmatrix}^{-1} \begin{pmatrix} f\circ\gamma^I(1) \\  f\circ\gamma^J(1)  \end{pmatrix}. 
\end{equation}
	  Here   $I,J$ are arbitrary elements in $\mathbb{S}(\Omega,\gamma)$ with  the condition that   $I \neq J$.	
\end{thm}

\begin{proof}
This theorem is a direct consequence of  Propositions \ref{pr-spsfeq},  \ref{eq:234} and Lemma \ref{lem-1}.
	\end{proof}

\begin{thm}
For a self-stem-preserving subset  $\Omega$ of  $\mathbb{H}_s^n$,  the set of  path-slice functions  $\mathcal{PS}(\Omega)$ forms an associative unitary real algebra with the usual addition and the  $*$-product. 
\end{thm}

\begin{proof}
	Verifying that  $(\mathcal{PS}(\Omega),+)$ is a real vector space is straightforward. The ring structure of  $(\mathcal{PS}(\Omega),+,*)$  with the unitary element
 $1_{\Omega}:\Omega\rightarrow\mathbb{H}$,  mapping $q$ to $1$,  is evident from the associative property as shown in equation \eqref{eq-associative}. Additionally, the following equation
	\begin{equation*}
		(\lambda f)*g=\lambda(f*g)=f*(\lambda g) 
	\end{equation*}
for all  $f,g\in\mathcal{PS}(\Omega)$ demonstrates that $(\mathcal{PS}(\Omega),+,*)$ constitutes an associative unitary real algebra.
\end{proof}

We will now present examples of self-stem-preserving domains.

We consider the induced topology of $\mathbb{H}^n_s$ as a subset of  $\mathbb{H}^n$, which carries the Euclidean topology $\tau(\mathbb{H}^n)$. Specifically, the topology on 
  $\mathbb{H}^n_s$  is defined as

\begin{equation*}
	\tau(\mathbb{H}^n_s):=\{U\cap\mathbb{H}_s^n: U\in\tau(\mathbb{H}^n)\}.
\end{equation*}

  \begin{lem}\label{pr-slicedomain rpc}
  	Let \(\Omega \subset \mathbb{H}_s^n\) be a slice-domain with \(\Omega_{\mathbb{R}} \neq \varnothing\). Then \(\Omega\) is real-path-connected.
  \end{lem}
  
  \begin{proof}
  	Consider any point \(q \in \Omega\). This point resides in \(\mathbb{C}_I^n\) for some \(I \in \mathbb{S}\). If no path exists in \(\Omega_I\) connecting a point in \(\Omega_{\mathbb{R}}\) to \(q\), then \(q \notin \Omega_{\mathbb{R}}\). Let \(U\) be the connected component of \(\Omega_I\) that contains \(q\), considered within the topology \(\tau(\mathbb{C}_I)\). Then, \(U\) and \(\Omega \setminus U\) are two disjoint, non-empty slice-open sets, implying that \(\Omega\) is not slice-connected, which is a contradiction. Thus, there must be a path \(\gamma^I\) in \(\Omega_I\) from some point in \(\Omega_{\mathbb{R}}\) to \(q\), proving that \(\Omega\) is real-path-connected.
  \end{proof}

\begin{thm}\label{pr-selfsp}
If $\Omega\in\tau(\mathbb{H}_s^n)$ is a slice-domain with $\Omega_\mathbb{R}\neq\varnothing$,  then $\Omega$ is self-stem-preserving.
\end{thm}

\begin{proof}
	Lemma \ref{pr-slicedomain rpc} ensures that 
$\Omega$   is real-path-connected.

	(i) Since \(\Omega \in \tau(\mathbb{H}_s^n)\), there exists \(\mathcal{O} \in \tau(\mathbb{H}^n)\) such that \(\Omega = \mathcal{O} \cap \mathbb{H}_s^n\). For any  \(\gamma \in \mathscr{P}(\mathbb{C}^n, \Omega)\) and \(I \in \mathbb{S}(\Omega, \gamma)\), the closed path   \(\gamma([0,1])\)   in \(\mathcal{O}\) yields a positive distance 
	\begin{equation*}
		r := \mathrm{dist}_{\mathbb{H}^n}(\gamma([0,1]), \mathbb{H}^n \setminus \mathcal{O}) > 0.
	\end{equation*}
Choosing  \(J \in \mathbb{S} \setminus \{I\}\) with a sufficiently small magnitude of $|J-I|$ guarantees that
	\begin{equation*}
		\sup_{t \in [0,1]} |\gamma^I(t) - \gamma^J(t)| < r.
	\end{equation*}
	This implies $\gamma^J \subseteq \mathcal{O} \cap \mathbb{H}_s^n = \Omega$,  making  $J$ a member of $\mathbb{S}(\Omega, \gamma)$. Therefore,
	for every   $\gamma \in \mathscr{P}(\mathbb{C}^n, \Omega)$ we have 
	\begin{equation*}
		|\mathbb{S}(\Omega, \gamma)| > 1.
	\end{equation*}
	
	(ii) For paths  $\alpha, \beta$ in $\mathscr{P}(\mathbb{C}^n, \Omega)$ that share an element in their respective sets  $\mathbb{S}(\Omega, \alpha)$ and $\mathbb{S}(\Omega, \beta)$,
	 and end at the same point  \(\alpha(1) = \beta(1)\). Let \(I \in \mathbb{S}(\Omega, \alpha) \cap \mathbb{S}(\Omega, \beta)\). Choose \(J \in \mathbb{S} \setminus \{I\}\) with \(|J - I|\) small enough such that
	\begin{equation*}
		\sup_{t \in [0,1], \gamma \in \{\alpha, \beta\}} |\gamma^I(t) - \gamma^J(t)| < \mathrm{dist}_{\mathbb{H}^n}(\alpha([0,1]) \cup \beta([0,1]), \mathbb{H}^n \setminus \mathcal{O}),
	\end{equation*}
	where \(\mathcal{O}\)  is as previously defined.  This condition ensures that both 
	$\alpha^J$ and $\beta^J$ are within $\Omega$,   leading to
	\begin{equation*}
		|\mathbb{S}(\Omega, \alpha) \cap \mathbb{S}(\Omega, \beta)|  \geqslant 2.
	\end{equation*}
	 
Combining these results from (i) and (ii), we conclude that  \(\Omega\) is \(\Omega\)-stem-preserving
  and, therefore, self-stem-preserving.
\end{proof}

\begin{cor}\label{pr-selfsp-367}
If   $\Omega$ is  an euclidean domain  in $\mathbb{H}$ that  intersects with the real  axis and  $\Omega_I$ forms   a domain in $\mathbb{C}_I$ for every  $I\in\mathbb{S}$, then $\Omega$ is self-stem-preserving.
\end{cor}

\begin{proof}
This result follows immediately from Theorem \ref{pr-selfsp} in the case where $n=1$.  
	\end{proof}

\begin{defn}
A domain $\Omega$ within $\mathbb{H}_s^n$ is defined as weakly axially symmetric if, for any point $q = x + yI$ in $\Omega$, its conjugate $\overline{q} = x - yI$ also lies within $\Omega$.
\end{defn}

\begin{thm}  If a domain $\Omega$ within $\mathbb{H}_s^n$ is a weakly axially symmetric slice-domain and   $\Omega_{\mathbb{R}}$ is non-empty, then $\Omega$ is  a self-stem-preserving domain.
\end{thm}

\begin{proof}
	  Lemma \ref{pr-slicedomain rpc} confirms that $\Omega$ is real-path-connected.
	
	(i) Considering any path $\gamma$ from $\mathscr{P}(\mathbb{C}^n,\Omega)$, there exists an imaginary unit $I$ in $\mathbb{S}$ such that the path
	 $\gamma^I$ is included in $\Omega$. Given a weakly axially symmetry domain  $\Omega$, the conjugate path $\gamma^{-I}$ (equal to $\overline{\gamma^I}$) also lies within $\Omega$. Hence, both $I$ and its negative counterpart $-I$ are elements of $\mathbb{S}(\Omega,\gamma)$, ensuring that the size of this set is at least two, or $|\mathbb{S}(\Omega,\gamma)| \geqslant 2$.
	
	(ii) For any pair of paths $\alpha$ and $\beta$ in $\mathscr{P}_*^2(\mathbb{C}^n,\Omega)$ with the intersection $\mathbb{S}(\Omega,\alpha) \cap \mathbb{S}(\Omega,\beta)\neq\varnothing$, let's assume $I$ is an element in this intersection. Utilizing the rationale from (i), we can deduce that    $-I$, also belongs to this intersection. Consequently, the intersection $\mathbb{S}(\Omega,\alpha) \cap \mathbb{S}(\Omega,\beta)$ contains at least two distinct elements, implying $|\mathbb{S}(\Omega,\alpha) \cap \mathbb{S}(\Omega,\beta)| \geqslant 2$.
	\end{proof}

In conclusion, our study introduces the innovative concept of self-path-preserving slice domains, which effectively supplants the role traditionally held by axially symmetric domains in the field of slice analysis for quaternions. 
A key example of this is seen in every slice-connected Euclidean domain, which intersects the real axis, is real-path-connected, which inherently qualify as self-path-preserving slice domains. The defining characteristic of these domains is their ability to establish a one-to-one correspondence   between slice functions and path-stem functions, a fundamental aspect that enhances the analytical process.
Moreover, within a self-path-preserving slice-domain, the entirety of slice functions forms an associative algebra when combined using the $*$-product. 
Our forthcoming research will delve into the theory of slice regular functions operating over self-path-preserving slice domains. This exploration promises to expand our understanding and application of these functions in a more comprehensive and nuanced way, leveraging the unique properties of self-path-preserving slice domains to develop new insights and methodologies in the field.

\bibliographystyle{plain}
\bibliography{mybibfile}

\end{document}